\documentclass{amsart}
\usepackage{amsthm}
\usepackage{amstext}
\usepackage{amssymb}
\usepackage{hyperref}

\makeatletter
\numberwithin{equation}{section}
\numberwithin{figure}{section}
\theoremstyle{plain}
\newtheorem{thm}{Theorem}[section]
  \theoremstyle{remark}
  \newtheorem*{rem*}{Remark}

\allowdisplaybreaks
\usepackage[sort&compress,square,numbers]{natbib}

\newcommand\relphantom[1]{\mathrel{\phantom{#1}}}

\makeatother

\begin{document}

\title{ Some identities of Carlitz degenerate Bernoulli numbers and polynomials}

\author{Taekyun Kim}
\address{Department of Mathematics, Tianjin Polytechnic University,
Tianjin, China\\Department of Mathematics, Kwangwoon University,
Seoul 139-701, Republic of Korea} \email{tkkim@kw.ac.kr}

\author{Dae San Kim}
\address{Department of Mathematics, Sogang University, Seoul 121-742, Republic
of Korea}
\email{dskim@sogang.ac.kr}

\author{Hyuck-In Kwon}
\address{Department of Mathematics, Kwangwoon University, Seoul 139-701, Republic
of Korea} \email{sura@kw.ac.kr}

\keywords{Carlitz degenerate Bernoulli numbers and polynomials, degenerate Riemann zeta function}
\subjclass[2010]{11B68, 11B83, 11C08, 65D20, 65Q30, 65R20}

\begin{abstract}
In this paper, we study the Carlitz's degenerate Bernoulli numbers
and polynomials and give some formulae and identities related to those
numbers and polynomials.
\end{abstract}

\maketitle
\global\long\def\acl#1#2{\left\langle \left.#1\right|#2\right\rangle }
\global\long\def\acr#1#2{\left\langle #1\left|#2\right.\right\rangle }
\global\long\def\Li{\mathrm{Li}}
\global\long\def\Zp{\mathbb{Z}_{p}}

\section{Introduction}

As is well known, the ordinary Bernoulli polynomials are defined by
the generating function
\begin{equation}
\left(\frac{t}{e^{t}-1}\right)e^{xt}=\sum_{n=0}^{\infty}B_{n}
\left(x\right)\frac{t^{n}}{n!},\quad\left(\text{see
[1-20]}\right).\label{eq:1}
\end{equation}

When $x=0$, $B_{n}=B_{n}\left(0\right)$ are called the Bernoulli
numbers.

From (\ref{eq:1}), we note that
\begin{align}
B_{n}\left(x\right) & =\sum_{l=0}^{n}\binom{n}{l}B_{l}x^{n-l}\label{eq:2}\\
 & =d^{n-1}\sum_{a=0}^{d-1}B_{n}\left(\frac{a+x}{d}\right),\quad\left(n\ge0,\, d\in\mathbb{N}\right).\nonumber
\end{align}

It is easy to show that
\begin{equation}
\frac{t}{e^{t}-1}e^{t}-\frac{t}{e^{t}-1}=t.\label{eq:3}
\end{equation}

Thus, by (\ref{eq:1}) and (\ref{eq:3}), we get
\begin{equation}
\sum_{n=0}^{\infty}\left(B_{n}\left(1\right)-B_{n}\right)\frac{t^{n}}{n!}=t.\label{eq:4}
\end{equation}

By comparing the coefficients on the both sides, we have
\begin{equation}
B_{0}=1,\quad B_{n}\left(1\right)-B_{n}=\delta_{n,1},\quad\left(\text{see \cite{key-20}}\right),\label{eq:5}
\end{equation}
where $\delta_{n,k}$ is the Kronecker's symbol.

Let $\chi$ be a Dirichlet character with conductor $d\in\mathbb{N}$.
Then, the generalized Bernoulli numbers attached to $\chi$ are defined
by the generating function
\begin{equation}
\frac{t}{e^{dt}-1}\sum_{a=0}^{d-1}\chi\left(a\right)e^{at}=\sum_{n=0}^{\infty}B_{n,\chi}\frac{t^{n}}{n!},\quad\left(\text{see \cite{key-10,key-14,key-20}}\right).\label{eq:6}
\end{equation}

Thus, by (\ref{eq:6}), we get
\[
B_{n,\chi}=d^{n-1}\sum_{a=0}^{d-1}\chi\left(a\right)B_{n}\left(\frac{a}{d}\right).
\]

For $\lambda\in\mathbb{C}$, L. Carlitz defined the degenerate Bernoulli
polynomials as follows:
\begin{equation}
\frac{t}{\left(1+\lambda t\right)^{\frac{1}{\lambda}}-1}\left(1+\lambda t\right)^{\frac{x}{\lambda}}=\sum_{n=0}^{\infty}\beta_{n}\left(x\mid\lambda\right)\frac{t^{n}}{n!},\quad\left(\text{see \cite{key-3,key-4}}\right).\label{eq:7}
\end{equation}

When $x=0$, $\beta_{n}\left(\lambda\right)=\beta_{n}\left(0\mid\lambda\right)$
are called the degenerate Bernoulli numbers.

By (\ref{eq:7}), we easily get
\begin{align}
 & \relphantom{=}\sum_{n=0}^{\infty}\lim_{\lambda\rightarrow0}\beta_{n}\left(x\mid\lambda\right)\frac{t^{n}}{n!}\label{eq:8}\\
 & =\lim_{\lambda\rightarrow0}\frac{t}{\left(1+\lambda t\right)^{\frac{1}{\lambda}}-1}\left(1+\lambda t\right)^{\frac{x}{\lambda}}\nonumber \\
 & =\frac{t}{e^{t}-1}e^{xt}\nonumber \\
 & =\sum_{n=0}^{\infty}B_{n}\left(x\right)\frac{t^{n}}{n!}.\nonumber
\end{align}

Thus, by (\ref{eq:8}), we see
\[
\lim_{\lambda\rightarrow0}\beta_{n}\left(x\mid\lambda\right)=B_{n}\left(x\right),\quad\left(n\ge0\right).
\]

In this paper, we study the properties of degenerate Bernoulli numbers
and polynomials and give some formulae and identities related to those
numbers and polynomials.

\section{Degenerate Bernoulli numbers and polynomials}

We easily see that
\begin{equation}
\frac{t}{\left(1+\lambda t\right)^{\frac{1}{\lambda}}-1}\left(1+\lambda t\right)^{\frac{1}{\lambda}}-\frac{t}{\left(1+\lambda t\right)^{\frac{1}{\lambda}}-1}=t.\label{eq:9}
\end{equation}

From (\ref{eq:7}) and (\ref{eq:9}), we have
\begin{equation}
\sum_{n=0}^{\infty}\left\{ \beta_{n}\left(1\mid\lambda\right)-\beta_{n}\left(\lambda\right)\right\} \frac{t^{n}}{n!}=t.\label{eq:10}
\end{equation}

By comparing the coefficients on the both sides of (\ref{eq:10}),
we get
\begin{equation}
\beta_{n}\left(1\mid\lambda\right)-\beta_{n}\left(\lambda\right)=\delta_{1,n},\quad\beta_{0}\left(\lambda\right)=1,\quad\left(n\in\mathbb{N}\right).\label{eq:11}
\end{equation}

Note that equation (\ref{eq:11}) is the $\lambda$-analogue of (\ref{eq:5}).

From (\ref{eq:7}), we can derive the following equation:
\begin{align}
t\left(1+\lambda t\right)^{\frac{x}{\lambda}} & =\left(\left(1+\lambda t\right)^{\frac{1}{\lambda}}-1\right)\sum_{m=0}^{\infty}\beta_{m}\left(x\mid\lambda\right)\frac{t^{m}}{m!}\label{eq:12}\\
 & =\left(\sum_{l=1}^{\infty}\left(1\mid\lambda\right)_{l}\frac{t^{l}}{l!}\right)\left(\sum_{m=0}^{\infty}\beta_{m}\left(x\mid\lambda\right)\frac{t^{m}}{m!}\right),\nonumber
\end{align}
where
\[
\left(x\mid\lambda\right)_{n}=x\left(x-\lambda\right)\cdots\left(x-\lambda\left(n-1\right)\right)=\lambda^{n}\left(\frac{x}{\lambda}\right)_{n}=\lambda^{n}\sum_{l=0}^{n}S_{1}\left(n,l\right)\lambda^{-l}x^{l}.
\]

Thus, by (\ref{eq:12}), we get
\begin{equation}
\sum_{n=0}^{\infty}\left(x\mid\lambda\right)_{n}\frac{t^{n}}{n!}=\sum_{n=0}^{\infty}\left(\sum_{l=0}^{n}\frac{\left(1\mid\lambda\right)_{l+1}}{l+1}\binom{n}{l}\beta_{n-l}\left(x\mid\lambda\right)\right)\frac{t^{n}}{n!}.\label{eq:13}
\end{equation}

By comparing the coefficients on both sides of (\ref{eq:13}),
we get
\begin{equation}
\left(x\mid\lambda\right)_{n}=\sum_{l=0}^{n}\frac{\left(1\mid\lambda\right)_{l+1}}{l+1}\binom{n}{l}\beta_{n-l}\left(x\mid\lambda\right),\quad\left(n\ge0\right).\label{eq:14}
\end{equation}

Note that

\[
x^{n}=\lim_{\lambda\rightarrow0}\left(x\mid\lambda\right)_{n}=\sum_{l=0}^{n}\binom{n}{l}\frac{B_{n-l}\left(x\right)}{l+1}.
\]

On the other hand,
\begin{align}
\sum_{n=0}^{\infty}\beta_{n}\left(x\mid\lambda\right)\frac{t^{n}}{n!} & =\left(\frac{t}{\left(1+\lambda t\right)^{\frac{1}{\lambda}}-1}\right)\left(1+\lambda t\right)^{\frac{x}{\lambda}}\label{eq:15}\\
 & =\left(\sum_{l=0}^{\infty}\beta_{l}\left(\lambda\right)\frac{t^{l}}{l!}\right)\left(\sum_{m=0}^{\infty}\left(x\mid\lambda\right)_{m}\frac{t^{m}}{m!}\right)\nonumber \\
 & =\sum_{n=0}^{\infty}\left(\sum_{l=0}^{n}\beta_{l}\left(\lambda\right)\binom{n}{l}\left(x\mid\lambda\right)_{n-l}\right)\frac{t^{n}}{n!}\nonumber
\end{align}

By comparing the coefficients on both sides of (\ref{eq:15}),
we have
\begin{equation}
\beta_{n}\left(x\mid\lambda\right)=\sum_{l=0}^{n}\binom{n}{l}\beta_{l}\left(\lambda\right)\left(x\mid\lambda\right)_{n-l},\quad\left(n\ge0\right).\label{eq:16}
\end{equation}

Therefore, by (\ref{eq:11}), (\ref{eq:14}) and (\ref{eq:16}), we
obtain the following theorem.
\begin{thm}
\label{thm:1} For $n\ge0$, we have
\begin{align*}
\beta_{n}\left(x\mid\lambda\right) & =\sum_{l=0}^{n}\binom{n}{l}\beta_{l}\left(\lambda\right)\left(x\mid\lambda\right)_{n-l},\\
\left(x\mid\lambda\right)_{n} & =\sum_{l=0}^{n}\frac{\left(1\mid\lambda\right)_{l+1}}{l+1}\binom{n}{l}\beta_{n-l}\left(x\mid\lambda\right),
\end{align*}
and
\[
\beta_{0}\left(\lambda\right)=1,\quad\beta_{n}\left(1\mid\lambda\right)-\beta_{n}\left(\lambda\right)=\delta_{1,n}.
\]

\end{thm}
From (\ref{eq:7}), we can derive the following equation:
\begin{align}
 & \relphantom{=}{}\frac{t}{\left(1+\lambda t\right)^{\frac{1}{\lambda}}-1}\left(1+\lambda t\right)^{\frac{x}{\lambda}}\label{eq:17}\\
 & =\frac{t}{\left(1+\lambda t\right)^{d/\lambda}-1}\sum_{a=0}^{d-1}\left(1+\lambda t\right)^{\frac{a+x}{\lambda}}\nonumber \\
 & =\frac{1}{d}\left(\frac{dt}{\left(1+\frac{\lambda}{d}dt\right)^{\frac{d}{\lambda}}-1}\right)\sum_{a=0}^{d-1}\left(1+\frac{\lambda}{d}dt\right)^{\frac{d}{\lambda}\frac{a+x}{d}}\nonumber \\
 & =\frac{1}{d}\sum_{a=0}^{d-1}\sum_{n=0}^{\infty}d^{n}\beta_{n}\left(\left.\frac{a+x}{d}\right|\frac{\lambda}{d}\right)\frac{t^{n}}{n!}\nonumber \\
 & =\sum_{n=0}^{\infty}\left\{ d^{n-1}\sum_{a=0}^{d-1}\beta_{n}\left(\left.\frac{a+x}{d}\right|\frac{\lambda}{d}\right)\right\} \frac{t^{n}}{n!},\quad\left(d\in\mathbb{N}\right).\nonumber
\end{align}

Therefore, by (\ref{eq:7}) and (\ref{eq:17}), we obtain the following
theorem.
\begin{thm}
\label{thm:2} For $n\ge0$, $d\in\mathbb{N}$, we have
\[
\beta_{n}\left(x\mid\lambda\right)=d^{n-1}\sum_{a=0}^{d-1}\beta_{n}\left(\left.\frac{a+x}{d}\right|\frac{\lambda}{d}\right).
\]
\end{thm}
\begin{rem*}
Theorem (\ref{thm:2}) is the $\lambda$-analogue of (\ref{eq:2}).
That is,
\[
B_{n}\left(x\right)=\lim_{\lambda\rightarrow0}\beta_{n}\left(x\mid\lambda\right)=d^{n-1}\sum_{a=0}^{d-1}B_{n}\left(\frac{a+x}{d}\right),\quad\left(d\in\mathbb{N}\right).
\]

\end{rem*}
We observe that
\begin{align}
 & \relphantom{=}{}\frac{t}{\left(1+\lambda t\right)^{\frac{1}{\lambda}}-1}\left(1+\lambda t\right)^{\frac{n}{\lambda}}-\frac{t}{\left(1+\lambda t\right)^{\frac{1}{\lambda}}-1}\label{eq:18}\\
 & =\frac{t}{\left(1+\lambda t\right)^{\frac{1}{\lambda}}-1}\left(\left(1+\lambda t\right)^{\frac{n}{\lambda}}-1\right)\nonumber \\
 & =\frac{t}{\left(1+\lambda t\right)^{\frac{1}{\lambda}}-1}\left(\left(1+\lambda t\right)^{\frac{1}{\lambda}}-1\right)\left(1+\left(1+\lambda t\right)^{\frac{1}{\lambda}}+\cdots+\left(1+\lambda t\right)^{\frac{n-1}{\lambda}}\right)\nonumber \\
 & =t\sum_{l=0}^{n-1}\left(1+\lambda t\right)^{\frac{l}{\lambda}}\nonumber \\
 & =t\sum_{m=0}^{\infty}\left(\sum_{l=0}^{n-1}\left(l\mid\lambda\right)_{m}\right)\frac{t^{m}}{m!},\quad\left(n\in\mathbb{N}\right).\nonumber
\end{align}

On the other hand,
\begin{align}
\frac{t}{\left(1+\lambda t\right)^{\frac{1}{\lambda}}-1}\left(1+\lambda t\right)^{\frac{n}{\lambda}}-\frac{t}{\left(1+\lambda t\right)^{\frac{1}{\lambda}}-1} & =\sum_{m=0}^{\infty}\left\{ \beta_{m}\left(n\mid\lambda\right)-\beta_{m}\left(\lambda\right)\right\} \frac{t^{m}}{m!}\label{eq:19}\\
 & =t\sum_{m=0}^{\infty}\left\{ \frac{\beta_{m+1}\left(n\mid\lambda\right)-\beta_{m+1}\left(\lambda\right)}{m+1}\right\} \frac{t^{m}}{m!}.\nonumber
\end{align}

By (\ref{eq:18}) and (\ref{eq:19}), we get
\begin{equation}
\sum_{m=0}^{\infty}\left(\sum_{l=0}^{n-1}\left(l\mid\lambda\right)_{m}\right)\frac{t^{m}}{m!}=\sum_{m=0}^{\infty}\left(\frac{\beta_{m+1}\left(n\mid\lambda\right)-\beta_{m+1}\left(\lambda\right)}{m+1}\right)\frac{t^{m}}{m!},\quad\left(n\in\mathbb{N}\right).\label{eq:20}
\end{equation}

Therefore, by comparing the coefficients on both sides of (\ref{eq:20}),
we obtain the following theorem.
\begin{thm}
\label{thm:3} For $n\in\mathbb{N}$, $m\ge0$, we have
\[
\sum_{l=0}^{n-1}\left(l\mid\lambda\right)_{m}=\frac{1}{m+1}\left\{ \beta_{m+1}\left(n\mid\lambda\right)-\beta_{m+1}\left(\lambda\right)\right\} .
\]

\end{thm}
By replacing $t$ by $\frac{1}{\lambda}\log\left(1+\lambda t\right)$
in (\ref{eq:1}), we get
\begin{align}
 & \relphantom{=}{}\frac{\log\left(1+\lambda t\right)^{\frac{1}{\lambda}}}{\left(1+\lambda t\right)^{\frac{1}{\lambda}}-1}\left(1+\lambda t\right)^{\frac{x}{\lambda}}\label{eq:21}\\
 & =\sum_{n=0}^{\infty}B_{n}\left(x\right)\lambda^{-n}\frac{1}{n!}\left(\log\left(1+\lambda t\right)\right)^{n}\nonumber \\
 & =\sum_{n=0}^{\infty}\left(\sum_{m=0}^{n}B_{m}\left(x\right)\lambda^{n-m}S_{1}\left(n,m\right)\right)\frac{t^{n}}{n!}.\nonumber
\end{align}

On the other hand,
\begin{align}
 & \relphantom{=}{}\frac{\log\left(1+\lambda t\right)^{\frac{1}{\lambda}}}{\left(1+\lambda t\right)^{\frac{1}{\lambda}}-1}\left(1+\lambda t\right)^{\frac{x}{\lambda}}\label{eq:22}\\
 & =\left(\frac{\log\left(1+\lambda t\right)}{\lambda t}\right)\left(\frac{t}{\left(1+\lambda t\right)^{\frac{1}{\lambda}}-1}\left(1+\lambda t\right)^{\frac{x}{\lambda}}\right)\nonumber \\
 & =\left(\sum_{l=0}^{\infty}\frac{\left(-1\right)^{l}}{l+1}\lambda^{l}t^{l}\right)\left(\sum_{m=0}^{\infty}\beta_{m}\left(x\mid\lambda\right)\frac{t^{m}}{m!}\right)\nonumber \\
 & =\sum_{n=0}^{\infty}\left(\sum_{l=0}^{n}\frac{\left(-1\right)^{l}\lambda^{l}}{l+1}\frac{\beta_{n-l}\left(x\mid\lambda\right)n!}{\left(n-l\right)!}\right)\frac{t^{n}}{n!}\nonumber \\
 & =\sum_{n=0}^{\infty}\left(\sum_{l=0}^{n}\frac{l!}{l+1}\left(-1\right)^{l}\lambda^{l}\binom{n}{l}\beta_{n-l}\left(x\mid\lambda\right)\right)\frac{t^{n}}{n!}.\nonumber
\end{align}

Therefore, by (\ref{eq:21}) and (\ref{eq:22}), we obtain the following
theorem.
\begin{thm}
\label{thm:4} For $n\ge0$, we have
\[
\sum_{m=0}^{n}B_{m}\left(x\right)\lambda^{n-m}S_{1}\left(n,m\right)=\sum_{l=0}^{n}\frac{l!}{l+1}\binom{n}{l}\left(-1\right)^{l}\lambda^{l}\beta_{n-l}\left(x\mid\lambda\right),
\]
where $S_{1}\left(n,m\right)$ is the Stirling number of the first
kind.
\end{thm}
By replacing $t$ by $\frac{1}{\lambda}\left(e^{\lambda t}-1\right)$
in (\ref{eq:7}), we get
\begin{align}
 & \relphantom{=}{}\frac{1}{\lambda}\left(\frac{e^{\lambda t}-1}{e^{t}-1}\right)e^{xt}\label{eq:23}\\
 & =\sum_{m=0}^{\infty}\beta_{m}\left(x\mid\lambda\right)\frac{1}{m!}\lambda^{-m}\left(e^{\lambda t}-1\right)^{m}\nonumber \\
 & =\sum_{m=0}^{\infty}\beta_{m}\left(x\mid\lambda\right)\lambda^{-m}\sum_{n=m}^{\infty}S_{2}\left(n,m\right)\frac{\lambda^{n}t^{n}}{n!}\nonumber \\
 & =\sum_{n=0}^{\infty}\left(\sum_{m=0}^{n}\lambda^{n-m}S_{2}\left(n,m\right)\beta_{m}\left(x\mid\lambda\right)\right)\frac{t^{n}}{n!},\nonumber
\end{align}
where $S_{2}\left(n,m\right)$ is the Stirling number of the second
kind.

On the other hand,
\begin{align}
 & \relphantom{=}{}\frac{1}{\lambda}\left(\frac{e^{\lambda t}-1}{e^{t}-1}\right)e^{xt}\label{eq:24}\\
 & =\frac{1}{\lambda t}\left(\frac{t}{e^{t}-1}\right)\left(e^{\left(x+\lambda\right)t}-e^{xt}\right)\nonumber \\
 & =\frac{1}{\lambda t}\sum_{n=0}^{\infty}\left\{ B_{n}\left(x+\lambda\right)-B_{n}\left(x\right)\right\} \frac{t^{n}}{n!}\nonumber \\
 & =\frac{1}{\lambda}\sum_{n=0}^{\infty}\left\{ \frac{B_{n+1}\left(x+\lambda\right)-B_{n+1}\left(x\right)}{n+1}\right\} \frac{t^{n}}{n!}.\nonumber
\end{align}

From (\ref{eq:23}) and (\ref{eq:24}), we have
\begin{equation}
\frac{B_{n+1}\left(x+\lambda\right)-B_{n+1}\left(x\right)}{n+1}=\sum_{m=0}^{n}S_{2}\left(n,m\right)\lambda^{n-m+1}\beta_{m}\left(x\mid\lambda\right).\label{eq:25}
\end{equation}

Therefore, by (\ref{eq:25}), we obtain the following theorem.
\begin{thm}
\label{thm:5} For $n\ge0$, we have
\[
\frac{B_{n+1}\left(x+\lambda\right)-B_{n+1}\left(x\right)}{n+1}=\sum_{m=0}^{n}S_{2}\left(n,m\right)\lambda^{n-m+1}\beta_{m}\left(x\mid\lambda\right).
\]
\end{thm}
\begin{rem*}
From Theorem \ref{thm:3}, we note that
\begin{align*}
\sum_{l=0}^{n-1}l^{m} & =\lim_{\lambda\rightarrow0}\sum_{l=0}^{n-1}\left(l\mid\lambda\right)_{m}=\lim_{\lambda\rightarrow0}\frac{\beta_{m+1}\left(n\mid\lambda\right)-\beta_{m+1}\left(\lambda\right)}{m+1}\\
 & =\frac{B_{m+1}\left(n\right)-B_{m+1}}{m+1},\quad\left(m\ge0,\, n\in\mathbb{N}\right).
\end{align*}

\end{rem*}
For $s\in\mathbb{C}\setminus\left\{ 1\right\} $, we define the degenerate
Riemann zeta function as follows:
\begin{equation}
\zeta\left(s,x\mid\lambda\right)=\frac{1}{\Gamma\left(s\right)}\int_{0}^{\infty}\frac{\left(1+\lambda t\right)^{-\frac{x}{\lambda}}}{1-\left(1+\lambda t\right)^{-\frac{1}{\lambda}}}t^{s-1}dt,\label{eq:26}
\end{equation}
where $x\neq0,-1,-2,\dots$.

From (\ref{eq:26}), we note that
\[
\lim_{\lambda\rightarrow0}\zeta\left(s,x\mid\lambda\right)=\zeta\left(s,x\right)=\sum_{n=0}^{\infty}\frac{1}{\left(n+x\right)^{s}},
\]
where $x\neq0,-1,-2,\dots$.

By Laurent series and (\ref{eq:26}), we obtain the following theorem.
\begin{thm}
\label{thm:6} For $n\in\mathbb{N}$, we have
\[
\zeta\left(1-n,x\mid\lambda\right)=-\frac{\beta_{n}\left(x\mid\lambda\right)}{n}.
\]

\end{thm}
For $d\in\mathbb{N}$, let $\chi$ be a Dirichlet character with
conductor $d$. Then, we define the generalized degenerate Bernoulli
numbers attached to $\chi$ as follows:
\begin{equation}
\frac{t}{\left(1+\lambda t\right)^{d/\lambda}-1}\sum_{a=0}^{d-1}\chi\left(a\right)\left(1+\lambda t\right)^{\frac{a}{\lambda}}=\sum_{n=0}^{\infty}\beta_{n,\chi}\left(\lambda\right)\frac{t^{n}}{n!}.\label{eq:27}
\end{equation}

Then, by (\ref{eq:27}), we get
\begin{align}
 & \relphantom{=}\sum_{n=0}^{\infty}\beta_{n,\chi}\left(\lambda\right)\frac{t^{n}}{n!}\label{eq:28}\\
 & =\frac{1}{d}\sum_{a=0}^{d-1}\chi\left(a\right)\frac{dt}{\left(1+\lambda t\right)^{d/\lambda}-1}\left(1+\lambda t\right)^{\frac{a}{\lambda}}\nonumber \\
 & =\frac{1}{d}\sum_{a=0}^{d-1}\chi\left(a\right)\frac{dt}{\left(1+\frac{\lambda}{d}dt\right)^{d/\lambda}-1}\left(1+\frac{\lambda}{d}dt\right)^{\frac{a}{d}\cdot\frac{d}{\lambda}}\nonumber \\
 & =\sum_{n=0}^{\infty}\left(d^{n-1}\sum_{a=0}^{d-1}\chi\left(a\right)\beta_{n}\left(\left.\frac{a}{d}\right|\frac{\lambda}{d}\right)\right)\frac{t^{n}}{n!}.\nonumber
\end{align}

Therefore, by (\ref{eq:28}), we obtain the following theorem,.
\begin{thm}
For $n\ge0$, $d\in\mathbb{N}$, we have

\[
\beta_{n,\chi}\left(\lambda\right)=d^{n-1}\sum_{a=0}^{d-1}\chi\left(a\right)\beta_{n}\left(\left.\frac{a}{d}\right|\frac{\lambda}{d}\right).
\]

\end{thm}

\section{Further Remark}

Let $p$ be a fixed prime number. Throughout this section, $\mathbb{Z}_{p}$,
$\mathbb{Q}_{p}$ and $\mathbb{C}_{p}$ will denote the ring of $p$-adic
integers, the field of $p$-adic rational numbers and the completion
of the algebraic closure of $\mathbb{Q}_{p}$. The $p$-adic norm
is normalized as $\left|p\right|_{p}=\frac{1}{p}$. For $\lambda,t\in\mathbb{C}_{p}$
with $\left|\lambda t\right|_{p}<p^{-\frac{1}{p-1}}$, the degenerate
Bernoulli polynomials are given by the generating function to be
\[
\frac{t}{\left(1+\lambda t\right)^{\frac{1}{\lambda}}-1}\left(1+\lambda t\right)^{\frac{x}{\lambda}}=\sum_{n=0}^{\infty}\beta_{n}\left(x\mid\lambda\right)\frac{t^{n}}{n!}.
\]

Let $d$ be a positive integer. Then, we define
\begin{align*}
X & =\lim_{\stackrel{\longleftarrow}{N}}\left(\mathbb{Z}/dp^{N}\mathbb{Z}\right);\\
a+dp^{N}\mathbb{Z}_{p} & =\left\{ x\in X\left|x\equiv a\pmod{dp^{N}}\right.\right\} ;\\
X^{*} & =\bigcup_{\substack{0<a<dp\\
p\nmid a
}
}a+dp\mathbb{Z}_{p}.
\end{align*}

We shall usually take $0\le a<dp^{N}$ when we write $a+dp^{N}\mathbb{Z}_{p}$.
Now, we will use Theorem \ref{thm:2} to prove a $p$-adic distribution result.
\begin{thm}
\label{thm:8} For $k\ge0$, let $\mu_{k,\beta}$ be defined by
\begin{equation}
\mu_{k,\beta}\left(a+dp^{N}\mathbb{Z}_{p}\right)=\left(dp^{N}\right)^{k-1}\beta_{k}\left(\left.\frac{a}{dp^{N}}\right|\frac{\lambda}{dp^{N}}\right).\label{eq:29}
\end{equation}
Then $\mu_{k,\beta}$ extends to a $\mathbb{C}_{p}$-valued distribution
on the compact open sets $U\subset X$. \end{thm}
\begin{proof}
It is enough to show that
\[
\sum_{i=0}^{p-1}\mu_{k,\beta}\left(a+idp^{N}+dp^{N+1}\mathbb{Z}_{p}\right)=\mu_{k,\beta}\left(a+dp^{N}\mathbb{Z}_{p}\right).
\]
Indeed, by (\ref{eq:29}), we get
\begin{align*}
 & \relphantom{=}{}\sum_{i=0}^{p-1}\mu_{k,\beta}\left(a+idp^{N}+dp^{N+1}\mathbb{Z}_{p}\right)\\
 & =\left(dp^{N+1}\right)^{k-1}\sum_{i=0}^{p-1}\beta_{k}\left(\left.\frac{a+idp^{N}}{p}\right|\frac{\lambda}{dp^{N+1}}\right)\\
 & =\left(dp^{N}\right)^{k-1}p^{k-1}\sum_{i=0}^{p-1}\beta_{k}\left(\left.\frac{\frac{a}{dp^{N}}+i}{p}\right|\frac{\frac{\lambda}{dp^{N}}}{p}\right)\\
 & =\left(dp^{N}\right)^{k-1}\beta_{k}\left(\left.\frac{a}{dp^{N}}\right|\frac{\lambda}{dp^{N}}\right)\\
 & =\mu_{k,\beta}\left(a+dp^{N}\mathbb{Z}_{p}\right).
\end{align*}

\end{proof}
The locally constant function $\chi$ can be integrated against the
distribution $\mu_{k,\beta}$ defined by (\ref{eq:29}), and the result
is
\begin{align}
 & \relphantom{=}{}\int_{X}\chi\left(x\right)d\mu_{k,\beta}\left(x\right)\label{eq:30}\\
 & =\lim_{N\rightarrow\infty}\sum_{x=0}^{dp^{N}-1}\chi\left(x\right)\mu_{k,\beta}\left(x+dp^{N}\mathbb{Z}_{p}\right)\nonumber \\
 & =\lim_{N\rightarrow\infty}\left(dp^{N}\right)^{k-1}\sum_{x=0}^{dp^{N}-1}\chi\left(x\right)\beta_{k}\left(\left.\frac{x}{dp^{N}}\right|\frac{\lambda}{dp^{N}}\right)\nonumber \\
 & =\beta_{k,\chi}\left(\lambda\right).\nonumber
\end{align}

Thus, by (\ref{eq:30}), we get
\[
\int_{X}\chi\left(x\right)d\mu_{k,\beta}\left(x\right)=\beta_{k,\chi}\left(\lambda\right),\quad\left(k\ge0\right).
\]\\\\

\bibliographystyle{amsplain}
\providecommand{\bysame}{\leavevmode\hbox to3em{\hrulefill}\thinspace}
\providecommand{\MR}{\relax\ifhmode\unskip\space\fi MR }
\providecommand{\MRhref}[2]{%
  \href{http://www.ams.org/mathscinet-getitem?mr=#1}{#2}
}
\providecommand{\href}[2]{#2}

\end{document}